\documentclass[11pt]{article}

\usepackage[text={153mm,235mm},centering]{geometry}
\usepackage{enumerate}

\usepackage[mathscr]{eucal}
\usepackage[T1]{fontenc}
\usepackage[utf8]{inputenc}
\usepackage{lmodern}
\usepackage{amsmath,amssymb,amsthm,mathtools}
\usepackage{mathrsfs}
\usepackage{graphicx}
\usepackage[all]{xy}
\usepackage{enumitem}
\usepackage{hyperref}
\numberwithin{equation}{section}
\usepackage{tikz-cd}
\usepackage{stmaryrd}
\usepackage{microtype}
\usepackage{authblk}

\allowdisplaybreaks
\hypersetup{
  colorlinks=true,
  linkcolor=blue,
  citecolor=blue,
  urlcolor=blue
}

\newtheorem{theorem}{Theorem}[section]
\newtheorem{proposition}[theorem]{Proposition}
\newtheorem{corollary}[theorem]{Corollary}
\newtheorem{lemma}[theorem]{Lemma}
\newtheorem{conjecture}[theorem]{Conjecture}

\theoremstyle{definition}
\newtheorem{definition}[theorem]{Definition}
\newtheorem{remark}[theorem]{Remark}
\newtheorem{example}[theorem]{Example}

\newtheorem*{ack}{Acknowledgements}
\newtheorem*{convention}{Convention}

\newcommand{\kfield}{k}
\newcommand{\D}{\mathrm{D}}
\newcommand{\Perf}{\mathrm{Perf}}
\newcommand{\thick}{\mathrm{thick}}

\newcommand{\RHom}{\mathbf{R}\!\operatorname{Hom}}
\newcommand{\otL}{\otimes^{\mathbf L}}

\newcommand{\Adual}{A^{!}}
\newcommand{\Rep}{\mathrm{Rep}}

\newcommand{\ch}{\mathrm{ch}}

\title{Algebraic $K$-theory,
cohomotopy $K$-groups,
and Koszul duality
\footnotetext{Email: xjchen@scu.edu.cn, m.huang@newuu.uz, f.eshmatov@newuu.uz}}

\author[1,2]{Xiaojun Chen}

\author[1]{Farkhod Eshmatov}
\author[1]{Maozhou Huang}

\affil[1]{Department of Mathematics, New Uzbekistan University,
Tashkent, 100001 Uzbekistan}

\affil[2]{School of Mathematics, Sichuan University, Chengdu, Sichuan Province, 
610064 P.R. China}

\date{}

\begin{document}

\maketitle

\begin{abstract}
Let $A$ be an augmented differential graded 
algebra over a field $\kfield$ of 
characteristic zero, and let 
$\Adual=\RHom_A(\kfield,\kfield)$ be its 
Koszul dual algebra. 
Blumberg and Mandell showed that, under some finiteness conditions of $A$, the derived Koszul duality provides an equivalence between the $K$-theory $K(\thick_A(\kfield) )$ of the triangulated thick subcategory generated by $\kfield$ and the $K$-theory $K(\Adual)$ of the derived category of perfect $\Adual$-modules. 
Combining this equivalence 
with the Jones--Goodwillie Chern character and the Jones--McCleary isomorphism, we obtain that the $K$-groups $K_n(\thick_A(k))$ are a 
concrete candidate for Loday's conjectural 
 contravariant $K$-groups.



\end{abstract}

%\tableofcontents

\section{Introduction}\label{sect:intro}

Let $A$ be an augmented differential graded 
algebra over a field $\kfield$ of 
characteristic zero. We write $K_n(A)$ for the 
$n$-th algebraic $K$-group of $A$, defined as the 
$K$-group of a Waldhausen model for the derived category of perfect $A$-modules. There is a natural 
map, the Jones--Goodwillie Chern character
(see Goodwillie \cite{Goodwillie}
and McCarthy \cite{McCarthy}),
\[
\ch_n:K_n(A)\longrightarrow HC^{-}_n(A),
\]
from algebraic $K$-theory to negative cyclic 
homology.

In \cite[\S 11.4.4]{Loday}, Loday suggested that one should look for a contravariant version of algebraic $K$-theory, denoted informally by $K^n(A)$, together with a Chern character to cyclic cohomology and a pairing compatible with the canonical pairing between cyclic cohomology and negative cyclic homology. 
For Banach algebras, this 
philosophy 
is realized by Connes through Kasparov's 
$KK$-theory (see Connes 
\cite{Connes}). However, outside the Banach setting, to the best of the authors’ knowledge, a satisfactory general definition of 
a ``cohomotopical'' $K$-theory is not yet 
available.

In this note, we explain that, for dg algebras with certain finiteness conditions, 
the derived Koszul duality provides a natural 
candidate $K_{\mathrm{Kos}}(A)$ for
the contravariant $K$-theory.
More precisely, let \[\Adual:=\RHom_A(\kfield,\kfield)\] be the Koszul dual algebra of $A$, 
which is identified with
the linear dual of
the bar construction $B(A)$. 
Then the Koszul duality functors (see Beilinson-Ginzburg-Soergel
\cite{BGS})
$$
\mathbf{R}\mathrm{Hom}_A(-,\kfield):
\D(A)\rightleftarrows \D(\Adual) : 
\RHom_{\Adual}(-,\kfield)
$$
restrict to equivalences inverse to each other
\[
\thick_A(\kfield)
\simeq 
\Perf(\Adual).
\]
This implies an identification of 
$K$-theory spectra (see Blumberg-Mandell
\cite{BM})
\[
K_{\mathrm{Kos}}(A)\coloneqq K(\thick_A(\kfield))\simeq K(\Adual).
\]
On the other hand,
the famous result of
Jones and McCleary (see \cite{JM})
says that there is an isomorphism
\[
HC^{-}_\bullet(\Adual)\cong HC^{\bullet}(A).
\]
Consequently, the ordinary Jones--Goodwillie Chern character for $\Adual$ transports to a map from the $K$-theory of $\thick_A(\kfield)$ to the cyclic cohomology of $A$.

The main point can be summarized as follows, which should be read as an expository package 
rather than as a new result.

\begin{theorem}\label{thm:intro-main}
Assume that $A$ is an augmented dg algebra 
over $\kfield$ which is weakly Adams connected
(see
Definition~\text{\rm{\ref{def:finite}}} below).
Then:
\begin{enumerate}[label=\textup{(\alph*)},noitemsep,leftmargin=*,nosep]
\item there is an equivalence 
$
\thick_A(\kfield)\simeq \Perf(\Adual),
$
and consequently
\[
K^n_{\mathrm{Kos}}
(A) 
=K_n(\thick_A(\kfield))
\cong K_n(\Adual)
\quad \text{for}\quad n\ge 0;
\]

\item via the isomorphism
$HC^{-}_n(\Adual)\cong HC^n(A)$ for
$n\in\mathbb Z$,
the composite
    \[
   K^n_{\mathrm{Kos}}(A)= K_n (\thick_A(\kfield) )
    \xrightarrow{\ \sim\ } K_n(\Adual)
    \xrightarrow{\ch_n} HC^{-}_n(\Adual)
    \xrightarrow{\ \sim\ } HC^n(A)
    \]
defines a natural Chern character
    \[
    \ch_n^{\vee}:K_n (\thick_A(\kfield) )
    \longrightarrow HC^n(A)\quad\text{for}\quad n\in \mathbb Z.
    \]
\end{enumerate}
\end{theorem}

\begin{remark}The condition for $A$
is moderate. Almost all interesting examples
of Koszul algebras are weakly Adams connected;
in fact, they are even Adams connected, in
the sense of Definition~\ref{def:finite}
(see Lemma \ref{lem:comparison} for relations of these two notions).
Also, Theorem \ref{thm:intro-main}
holds for $A_\infty$-algebras, as studied in
\cite{LPWZ}.
\end{remark}

\begin{example}[The classical Koszul algebra]
Suppose $A$ is a Koszul or linear Koszul algebra in
the usual sense (see, for example,
\cite{BGS} and \cite{LV}), that is, the underlying
space of $A$ is $TV/(R)$, where $V$ is a 
finite-dimensional vector space, $TV$ is 
the free associative algebra generated by 
$V$, and $(R)$ is the two-sided ideal
generated by a subspace $R\subset V\otimes V$.
If we equip elements in $V$ with
Adams grading $1$, then $A$ is Adams
connected and hence strongly locally finite.
Thus Theorem~\ref{thm:intro-main} holds for all these Koszul algebras.
\end{example}

The note is organized as follows. In 
Section~\ref{sec:$K$-theory} we recall the 
Waldhausen model of $K$-theory that we use for 
dg algebras. Section~\ref{sec:cyclic} recalls 
cyclic homology, cyclic cohomology, their 
natural pairing, and the Jones--Goodwillie 
Chern character. Section~\ref{sec:koszul} 
contains the Koszul duality arguments. In 
Section~\ref{sec:loday} we formulate the 
resulting candidate for Loday's contravariant $K$-groups and discuss their relation with the category of finite-dimensional modules.

\begin{convention}
Throughout the paper,
$\kfield$ is a field of characteristic zero;
all dg algebras are unital, associative, and augmented over $\kfield$;
all modules are right dg modules unless explicitly stated otherwise.
Given a dg algebra $A$,
$\D(A)$ denotes the derived category of dg right $A$-modules.
\end{convention}

\begin{ack}%{Acknowledgements}

We would like
to thank Leilei Liu and Jieheng Zeng
for several helpful communications.
This work is supported by 
NSFC No. 12271377 and 12261131498.
\end{ack}

\section{Algebraic $K$-theory of dg algebras}\label{sec:$K$-theory}

In this section we recall the version of 
Waldhausen $K$-theory needed later. We do not 
attempt to work at maximum generality; the 
discussion is tailored to perfect dg modules 
over a dg algebra.
The interested reader may refer to
\cite{Waldhausen,Weibel} for more details.

\subsection{Waldhausen categories and the 
iterated $S_\bullet$-construction}

%We assume that the readers are familiarwith Waldhausen categories. 
Briefly, a Waldhause category 
is a category
together with two families of morphisms,
called cofibrations and weak equivalences,
satisfying several conditions.
Given a Waldhausen category,
Waldhausen introduced its
algebraic $K$-theory as follows.

Let $\operatorname{Ar}[n]$ be the category whose objects are pairs $(i,j)$ with $0\le i\le j\le n$ and whose morphisms are the obvious order-preserving maps.
For the category $\operatorname{Ar}[n_1,\ldots,n_q] \coloneqq \operatorname{Ar}[n_1]\times\cdots\times \operatorname{Ar}[n_q]$ and a functor $A:\operatorname{Ar}[n_1,\ldots,n_q] \to \mathcal{C},$ write $A_{i_1,j_1;\ldots;i_q,j_q} \coloneqq A((i_1,j_1),\ldots,(i_q,j_q)).$ 
For a Waldhausen category $\mathcal C$, let $S^{(q)}_{n_1,\ldots,n_q}\mathcal{C}$ denote the full subcategory consisting of functors in $\operatorname{Fun}(\operatorname{Ar}[n_1,\ldots,n_q],\mathcal C)$ satisfying the following conditions:
\begin{enumerate}[label=\textup{(\arabic*)},noitemsep,leftmargin=*,nosep]
\item $A_{i_1,j_1;\ldots;i_q,j_q} = \ast$ if some $i_k = j_k$.
\item Every map $A_{0,i_1;\ldots;0,i_q} \to A_{0,j_1;\ldots;0,j_q}$ is a cofibration. 
% the canonical map induced by the pushout $A_{0,i_1;\ldots;0,i_r+1;\ldots;0,i_q}\cup_{A_{0,i_1;\ldots;0,i_q}}A_{0,i_1;\ldots;0,i_s+1;\ldots;0,i_q}\to A_{0,i_1;\ldots;0,i_r+1;\ldots;0,i_s+1;\ldots;i_q,j_q}$ is a cofibration; 
\item For integers $1\leq r_1\leq \cdots \leq r_m\leq n$ and the cubic diagram given by
\begin{equation*}\begin{array}{cl}
&A((0,i_1),\ldots,(0,i_{r_1}+j_{r_{1}}),\ldots,(0,i_{r_2}+j_{r_{2}}),\ldots,(0,i_{r_m}+j_{r_{m}}),\ldots,(0,i_{n}))\\
&\text{for all }j_{r_1},\ldots,j_{r_m}=0,1\text{ except for }j_{r_1}= \cdots =j_{r_m} = 1,
\end{array}
\end{equation*}
the induced map from the colimit over this diagram to \[%A_{0,i_1;\ldots;0,i_{r_1}+1;\ldots;0,i_{r_2}+1;\ldots;0,i_{r_m}+1;\ldots;0,i_{n}}
A((0,i_1),\ldots,(0,i_{r_1}+1),\ldots,(0,i_{r_2}+1),\ldots,(0,i_{r_m}+1),\ldots,(0,i_{n}))
\] is a cofibration.
\item For any $1\leq r\leq q$ and $j_r\leq k\leq n_r$, the following induced map is isomorphic \[A_{i_1,j_1;\ldots;i_r,k;\ldots;i_q,j_q}/A_{i_1,j_1;\ldots;i_r,j_r;\ldots;i_q,j_q}\to A_{i_1,j_1;\ldots;j_r,k;\ldots;i_q,j_q}.\] 
\end{enumerate}
The categories $S^{(q)}_{n_1,\ldots,n_q}\mathcal C$ for all $n_1,\ldots,n_q$ assemble into a multisimplicial category $S^{(q)}_{\bullet,\ldots,\bullet}\mathcal C$, which is a Waldhausen category. For the subcategory $wS^{(q)}_{\bullet,\ldots,\bullet} \mathcal C$ of weak equivalences, the \emph{algebraic $K$-theory spectrum} $K(\mathcal C)$ of $\mathcal C$ has $q$-th space $|wS^{(q)}_{\bullet,\ldots,\bullet} \mathcal C|.$
Write $S_{\bullet}\mathcal{C} = S^{(q)}_{\bullet,\ldots,\bullet}\mathcal{C}$ for $q = 1$.
The $K$-groups of $\mathcal C$ are the homotopy groups $K_n(\mathcal C)\coloneqq \pi_{n+1}|wS_\bullet\mathcal C|.$

\subsection{Dg modules}

Let $A$ be a dg algebra and $\mathsf{Mod}\,A$ the 
category of dg right $A$-modules. 
We equip $\mathsf{Mod}\,A$ with the 
projective model structure: weak 
equivalences are 
quasi-isomorphisms, fibrations are 
degreewise surjections, and 
cofibrations are injective 
maps with dg projective cokernels. 
The derived category
of $\mathsf{Mod}\,A$
is the homotopy category
of $\mathsf{Mod}\,A$ localized
at quasi-isomorphisms, and
is denoted by $\mathrm{D}(A)$.
Given a dg $A$-module $M$, let $\mathrm{thick}_A(M)$ be the full subcategory of $\mathrm{D}(A)$ generated by $M$ that is closed under shifts and summands.

We have a Waldhausen category 
which models 
$\mathrm{thick}_A(M)$
in $\mathrm{D}(A)$ as follows:
it is 
the full subcategory of
$\mathsf{Mod}\,A$
consisting of cofibrant
dg $A$-modules whose
images in $\mathrm{D}(A)$
lie in $\mathrm{thick}_A(M)$.
In particular,
if $M=A$, we write
$\mathrm{thick}_A(A)$
as $\mathrm{Perf}(A)$, which is also the full triangulated subcategory of compact objects in $\D(A)$, and
the associated Waldhausen category
as $\mathcal M^c_A$;
if $M$ is the augmentation
module $k$, we write
the Waldhausen
category associated to
$\mathrm{thick}_A(k)$
as $\mathcal M_A(k)$.

\begin{definition}[{cf. 
\cite[Definition 4.4]{BM}}]
Given a dg algebra $A$,
\begin{enumerate}[label=\textup{(\arabic*)},noitemsep,leftmargin=*,nosep]
\item
the algebraic $K$-theory spectrum of $A$ is
\[
K(A):=K(\mathcal M^c_A)
\]
and its homotopy groups are denoted by 
$K_n(A)$,
and

\item the algebraic
$K$-theory spectrum of $\mathrm{thick}_A(k)$
is 
$$K(\mathrm{thick}_A(k)):=
K(\mathcal M_A(k))$$
and its homotopy groups are denoted by 
$K_n(\mathrm{thick}_A(k))$.
\end{enumerate}
\end{definition}

The above definition (1) 
agrees with the usual algebraic 
$K$-theory of a ring when $A$ is 
concentrated in degree zero, where
Quillen's and 
Waldhausen's constructions coincide.

\section{Cyclic homology, cyclic cohomology, 
and the Chern character}\label{sec:cyclic}

We collect some notions 
and facts about cyclic homology,
cyclic cohomology, and their pairings.
Details can be found in 
\cite{Goodwillie,Jones1987,Loday}.

\subsection{Mixed complexes}

Let $A$ be a dg algebra. Its normalized Hochschild 
chain complex is denoted by $C_\bullet(A)$, with 
Hochschild differential $b$ and Connes operator $B$. 
Together they form a mixed complex 
$(C_\bullet(A),b,B)$.

\begin{definition}
Let $u$ be a formal variable of degree $-2$. 
The \emph{negative cyclic complex}, 
\emph{periodic cyclic complex}, and 
\emph{cyclic complex} of $A$ are the complexes
\begin{align*}
CC^-_\bullet(A)&:=(C_\bullet(A)\llbracket 
u\rrbracket,b+uB),\\
CC^{\mathrm{per}}_\bullet(A)
&:=(C_\bullet(A)\llbracket 
u,u^{-1}\rrbracket,b+uB),\\
CC_\bullet(A)&:= (C_\bullet(A)\llbracket 
u,u^{-1}\rrbracket/uC_\bullet(A)\llbracket 
u\rrbracket,b+uB ).
\end{align*}
Their homology groups are denoted by 
$HC^-_\bullet(A)$, 
$HC^{\mathrm{per}}_\bullet(A)$, and 
$HC_\bullet(A)$ respectively.
\end{definition}

Dually, let 
$C^\bullet(A)=\operatorname{Hom}_{\kfield}
(C_\bullet(A),\kfield)$ be the normalized 
Hochschild cochain complex, with differentials 
$b^*$ and $B^*$ dual to $b$ and $B$.
Here we only recall the definition of cyclic
cohomology.

\begin{definition}
Let $v$ be a formal variable of degree $2$. 
The \emph{cyclic cochain complex} of $A$ is
\[
CC^{\bullet}(A):=(C^{\bullet}(A)\llbracket 
v\rrbracket,
b^*+vB^*).
\]
Its cohomology is the \emph{cyclic cohomology} 
of $A$, denoted by $HC^{\bullet}(A)$.
\end{definition}

\begin{remark}
When one works systematically with mixed 
complexes, all of these constructions are 
functorial. In particular, if $A$ and $A'$ are 
quasi-isomorphic as dg algebras, then their 
Hochschild, cyclic, and negative cyclic 
theories are canonically isomorphic.
\end{remark}

\subsection{The natural pairing}

There is a canonical pairing between cyclic
cochains and negative 
cyclic chains. At the 
chain level it is given by
\[
 \langle 
\sum_{j\ge 0}\varphi_j v^j,\,
\sum_{i\ge 0}x_i u^i
 \rangle
:=\sum_{i\ge 0}\varphi_i(x_i).
\]
A direct check shows that this is compatible 
with the differentials $b^*+vB^*$ and $b+uB$. 
Hence it descends to a pairing
\begin{equation}\label{eq:cyclic-pairing}
\langle -, -\rangle_{\mathrm{cyc}}:
HC^n(A)\times HC_n^{-}(A)\longrightarrow 
\kfield.
\end{equation}
This is the pairing used in Loday's 
formulation of the conjectural compatibility 
between homotopical and cohomotopical $K$-theories.

\subsection{The Jones--Goodwillie Chern character}

For a dg algebra $A$, the Chern character from its 
algebraic $K$-theory to the negative cyclic homology 
can be defined in several equivalent ways.
We use McCarthy's 
extension of Hochschild and cyclic theory to a
Waldhausen category, which is 
$\mathcal M_{A}^{c}$ in this paper 
(\cite[\S4]{McCarthy}).

\begin{theorem}[Jones--Goodwillie; 
McCarthy]\label{thm:JG}
Let $A$ be a dg algebra over $\kfield$. There 
is a natural map
from the $K$-theory spectrum
to the negative cyclic spectrum:
\[
\ch:K(A)\longrightarrow HC^{-}(A),
\]
whose homotopy groups give graded maps
\[
\ch_n:K_n(A)\longrightarrow HC^-_n(A),\qquad 
n\ge 0.
\]
These maps are functorial
and agree with the classical Chern character 
for ordinary algebras.
\end{theorem}

For our purposes, it is enough to know that 
the Chern character exists and is functorial,
which will be used in Section~\ref{sec:loday}.

\section{Derived Koszul duality}\label{sec:koszul} 

In this section, we collect
some facts on Koszul duality,
which we learned from 
\cite{BGS,BM,DGI,LPWZ}.
We take the framework of
Lu, Palmieri, Wu and Zhang \cite{LPWZ}.

\subsection{The Koszul dual algebra}

\emph{From now on}, we assume that $A$ is bigraded: for any $a\in A$, we write
$$\mathrm{deg}\, a=(\mathrm{deg}_1(a),
\mathrm{deg}_2(a))
\in\mathbb Z\times\mathbb Z,$$
where the first grading is called
the homological grading, and the second
is called the Adams grading. We write
$A=\bigoplus_{i,j}A^i_j$,
where $A^i_j$ is the degree $(i,j)$
component.
The differential has degree $(-1,0)$.
The Koszul sign rule is given as follows:
if we switch two elements with grading
$(i,s)$ and $(j,t)$, there is
a sign $(-1)^{ij}$.

\begin{definition}[{\cite[Definition 2.1]{LPWZ}}]
\label{def:finite}
Suppose $A$ is a dg algebra.
Let $I=\bigoplus_{i,j} I^i_j$ be its augmentation ideal.
\begin{enumerate}[label=\textup{(\arabic*)},noitemsep,leftmargin=*,nosep]
\item
$A$ is called \emph{strongly locally finite}
if $I$ satisfies the following:
\begin{enumerate}[label=\textup{(\alph*)},noitemsep,leftmargin=*,nosep]
\item Each
bihomogeneous component $I^i_j$ is finite dimensional;

\item $I_j^*=\bigoplus_i I^i_j=0$ either for all $j\le 0$ or for all $j\ge 0$;

\item For all $j$, $I_j^i=0$ either
for all $i$ big enough or small enough.
\end{enumerate}

\item $A$ is called \emph{Adams connected}
if $A$ is, with respect to the Adams grading,
either positively or negatively connected
graded and locally finite. In other words,
\begin{enumerate}[label=\textup{(\alph*)},noitemsep,leftmargin=*,nosep]
\item $I_j^*$ is finite-dimensional for all $j$;
\item $I_j^*=0$ either for all $j\le 0$ or for all
$j\ge 0$.
\end{enumerate}

\item $A$ is called \emph{weakly Adams connected}
if
\begin{enumerate}[label=\textup{(\alph*)},noitemsep,leftmargin=*,nosep]
\item the underlying space of the reduced
Hochschild complex $C_\bullet(A)$ is locally finite;
\item the only simple dg $A$-modules are $k$ and its
shifts;
\item $A$ is an inverse limit of a family of
finite-dimensional left dg $A$-bimodules.
\end{enumerate}
\end{enumerate}
\end{definition}

A bigraded algebra has a finer structure
than just viewing the algebra with
the homological grading.
For example, the polynomial algebra $k[x]$ is concentrated in degree $0$ and hence is not locally finite with respect to the homological grading; however, it is Adams connected (and hence locally finite) if we assign the Adams grading of $x$ to be $1$.
All results in the previous sections
remain true in this bigraded setting.

Recall that for an augmented dg algebra 
$A$, its \emph{Koszul dual coalgebra},
denoted by $A^{\textup{!`}}$,
is the bar construction 
$\mathsf{B}(A)$
of $A$, which is a dg
coalgebra, and
the \emph{Koszul dual algebra}
$A^!$
of $A$ is the linear
dual of $A^{\textup{!`}}$,
with the induced dg algebra
structure.
In fact,
$A^!$ is a dg algebra model
for
$\RHom_A(\kfield,\kfield)$.

A comparison of the 
notions in Definition
\ref{def:finite}, and
their relations to
Koszul duality, are given by the 
following.

\begin{lemma}[\text{\cite[Lemma~2.2]{LPWZ}}]
\label{lem:comparison}
For a dg algebra $A$, we have
\begin{enumerate}[label=\textup{(\arabic*)},noitemsep,leftmargin=*,nosep]
    \item Adams connected $\Rightarrow$ strongly locally finite $\Rightarrow$ weakly Adams connected;
    \item $A$ is weakly Adams connected $\Rightarrow$ $A^!$ is locally finite;
    \item $A$ is strongly locally finite $\Rightarrow$ $A^!$ is strongly locally finite;
    \item $A$ is Adams connected $\Rightarrow$ $A^!$ is Adams connected.
\end{enumerate}
\end{lemma}

The following result is nowadays
standard (see
\cite[Theorem 2.4]{LPWZ} 
for a proof).

\begin{theorem}\label{thm:doublecent}
Let $A$ be a dg algebra.
If $A^!$ is locally 
finite (for example, $A$ is weakly 
Adams connected), then
there is a canonical
quasi-isomorphism
$$A\stackrel{\simeq}
\longrightarrow
(A^!)^!.$$
\end{theorem}

\iffalse

In the above theorem, $A^!$ being locally 
finite is not so restrictive: for example,
if $A$ is Adams connected or strongly locally
finite, then $A^!$ is locally finite
(see \cite{LPWZ}, the paragraph after Theorem 2.4
for more details).

\fi

\subsection{The derived equivalences}

The augmentation module $\kfield$, which is a $(\Adual,A)$-bimodule, defines two derived functors 
\begin{equation}
\label{equivfunctors}
\RHom_A(-,\kfield):\D(A)
\longrightarrow \D(\Adual),
\qquad
\RHom_{\Adual}
(-,\kfield):\D(\Adual) \to \D(A).
\end{equation}
The next result is a version of 
derived Koszul duality. %that is relevant here. 
It may 
be viewed as a dg-algebraic analogue of the 
contravariant duality in the work of 
Beilinson--Ginzburg--Soergel \cite{BGS}, and it is 
closely related to the Morita/Koszul duality 
framework of Dwyer--Greenlees--Iyengar
\cite{DGI}
and 
Blumberg--Mandell \cite{BM}.

\begin{theorem}
\label{thm0:koszul-equivalence}
Let $A$ be an augmented dg
algebra. If $A$ is weakly Adams connected, then
\begin{enumerate}[label=\textup{(\arabic*)},noitemsep,leftmargin=*,nosep]
\item the functor
$$\mathbf{R}\mathrm{Hom}_A(-,k):
\mathrm{thick}_A(k)\to \mathrm{Perf}(A^!)$$
is an equivalence of triangulated categories;

\item 
if moreover 
$k\in\mathrm{Perf}(A)$, the functor
$$\mathbf{R}\mathrm{Hom}_A(-,k):
\mathrm{Perf}(A)\to \mathrm{thick}_{A^!}(k)$$
is an equivalence of triangulated categories.
\end{enumerate}

\end{theorem}

\begin{proof}
For the proof of (1), see
\cite[Proposition~4.11(d) \& Lemma~5.3(b)]{LPWZ}.
For the proof of (2), see
\cite[Theorem 5.5(c)]{LPWZ}.
\end{proof}

\begin{remark}
If $A$ itself is strongly 
locally finite, then
the condition that
$k\in\mathrm{Perf}(A)$
is equivalent to that
$\mathrm{dim}\,\mathrm{Ext}^\bullet_{A}(k,k)<\infty$;
see \cite[Corollary 6.2]{LPWZ}
for a proof.
\end{remark}

As a corollary, we obtain that the functors in (\ref{equivfunctors}) induce a $K$-theory spectra weak equivalence.
For any category $\mathcal C,$ let $\mathrm{Ar}[n_1,\ldots,n_q]\,\mathcal{C}$ be the category of functors $\mathrm{Ar}[n_1,\ldots,n_q]\to\mathcal{C}.$

\begin{corollary}[\text{cf.\ \cite[Lemma~4.7]{BM}}]
\label{cor:twoequiv}
Resume the conditions in Theorem
\text{\rm{\ref{thm0:koszul-equivalence}}}.
Then the functor $\RHom_A(-,\kfield)$ induces equivalences respectively
\begin{enumerate}[label=\textup{(\arabic*)},noitemsep,leftmargin=*,nosep]
\item 
between the full subcategory of the homotopy category of 
$\mathrm{Ar}[n_1,\ldots,n_q]\,\mathsf{Mod}\,A$ generated by 
objects in $S^{(q)}_{n_1,\ldots,n_q}\,\mathcal{M}_{A}(k)$
and the full subcategory of the homotopy category of 
$\mathrm{Ar}[n_1,\ldots,n_q]\,\mathsf{Mod}\,A^!$  
generated by objects in 
$S^{(q)}_{n_1,\ldots,n_q}\mathcal{M}_{A^!}^c$;

\item 
between the full subcategory of the homotopy category of $\mathrm{Ar}[n_1,\ldots,n_q]\,\mathsf{Mod}\,A$ generated by objects in $S^{(q)}_{n_1,\ldots,n_q}\,\mathcal{M}_{A}^c$ and the full subcategory of the homotopy category of $\mathrm{Ar}[n_1,\ldots,n_q]\,\mathsf{Mod}\,A^!$  generated by objects in $S^{(q)}_{n_1,\ldots,n_q}\,\mathcal{M}_{A^!}(k)$.
\end{enumerate}
\end{corollary}

\begin{proof}
We show (1). The proof of (2) is similar and omitted.
We claim that $S_{\bullet}$-construction preserves the equivalence between two categories generated by $\mathcal{M}_{A}(k)$ and $\mathcal{M}_{A^!}^c$ inside the homotopy categories.
Note that $\RHom_A(-,k)$ maps the zero object to itself, maps cofibrations to fibrations, preserves weak equivalences,  and maps homotopy pushout squares to homotopy pullback squares.
This equivalence follows from the facts below for a stable model category $\mathcal{C}$ \cite[Theorem~7.1.11]{Hovey}, where $\mathcal C$ may take $\mathrm{Ar}[n_q]\,\mathsf{Mod}\,A$ and $\mathrm{Ar}[n_q]\,\mathsf{Mod}\,A^!$ 
(see \cite[Example~2.4(i)]{SchShi}):
\begin{enumerate}[label=\textup{(\alph*)},noitemsep,leftmargin=*,nosep]
\item Every map in $\mathcal{C}$ is weak equivalent to a cofibration;
\item A square in $\mathcal{C}$ is homotopy pullback square in commuting square if and only if it is homotopy pushout if and only if it is weak equivalent to a pullback square of fibrations.
\end{enumerate}
Moreover, $S_{\bullet}$-construction preserves the equivalence between two categories generated by $S_{n_{q}}\mathcal{M}_{A}(k)$ and $S_{n_{q}}\mathcal{M}_{A^!}^c$ inside the homotopy categories.
The desired category equivalence follows from the induction. 
% For (ii), the argument is completely analogous and is left to the interested reader.
\end{proof}

\begin{lemma}[\text
{\cite[Lemma~4.6]{BM}}]Let $\mathcal{C}$ (resp. $\mathcal M$) be a closed Waldhausen category of cofibrant objects in a closed model category $\overline{\mathcal C}$ (resp. $\overline{\mathcal M}$) with cofibrations and weak equivalences from $\overline{\mathcal C}$ (resp. $\overline{\mathcal M}$).
Let $F:\mathcal{C}\to \overline{\mathcal M}$ be a contravariant functor which maps the zero object to itself, cofibrations to fibrations, and preserves the weak equivalences.
Assume that $F$ induces a category equivalence between two 
full subcategories in the homotopy categories 
$\mathrm{Ho}\,\overline{\mathcal C}$ and 
$\mathrm{Ho}\,\overline{\mathcal M}$  generated by $\mathcal C$ and $\mathcal M$ respectively.
Then $F$ induces a weak equivalence of the nerves of 
$w\mathcal C$ and $w\mathcal M.$ 
\end{lemma}

In the above lemma, taking $\mathcal{C}$ and $\mathcal{M}$ to be respectively $S^{(q)}_{n_1,\ldots,n_q}\,\mathcal{M}_{A}(k)$ and $S^{(q)}_{n_1,\ldots,n_q}\,\mathcal{M}_{A^!}^{c}$ (or $S^{(q)}_{n_1,\ldots,n_q}\,\mathcal{M}_{A}^c$ and $S^{(q)}_{n_1,\ldots,n_q}\,\mathcal{M}_{A^!}(k)$) in Corollray \ref{cor:twoequiv}, we then obtain the following.

\begin{corollary}[{Blumberg-Mandell 
\cite[Theorems~1.1 \& 1.2]{BM}}]\label{cor:$K$-equivalence}
Suppose $A$ is a dg algebra as in Theorem~\text{\rm{\ref{thm0:koszul-equivalence}}}.
Then there are two weak equivalences of $K$-theory spectra
\[K (\thick_A(\kfield) )\simeq K(\Adual),\qquad K(A)\simeq K (\thick_{\Adual}(\kfield) ).\]
In particular, $K_n (\thick_A(\kfield) )\cong K_n(\Adual)$ and $K_n(A)\cong K_n (\thick_{\Adual}(\kfield))$ for $n\ge 0.$
\end{corollary}

\begin{remark}
Consider the following two functors
\[
F:=-\otL_{A}\kfield:
\D(A)\longrightarrow \D(\Adual),
\qquad
G:=\RHom_{\Adual}(\kfield,-):\D(\Adual)
\longrightarrow \D(A),
\]
which form an adjoint pair.
It is also straightforward to show that
the adjunction 
$(F,G)$ restricts to inverse equivalences
\[
\Perf(A)\xrightarrow{\ \sim\ } 
\thick_{\Adual}(\kfield),
\qquad
\thick_{A}(\kfield)\xrightarrow{\ \sim\ } 
\Perf(\Adual),
\]
and hence the equivalence on the corresponding
$K$-theories.
However, notice that in this case the functor $A\mapsto K_n(\thick_{A}(\kfield))$ is covariant instead of being contravariant.
\end{remark}

\subsection{Cyclic theory of the Koszul dual}

Let $A$ and
$A^!$ as in the previous section.
The following result relates the negative cyclic homology of $A$ and cyclic cohomology of $A^!$:

\begin{proposition}[{Jones--McCleary \cite{JM}}]
\label{prop:JM}
Assume that $A^!$ is locally finite (for example, $A$ is weakly Adams connected). Then there is a natural isomorphism \[HC_n^{-}(A)\cong HC^n(\Adual)\text{ for any $n\in \mathbb Z$.}\]
\end{proposition}

\begin{proof}[Sketch of proof]
Jones--McCleary identify the 
mixed complex computing
cyclic 
homology of $A$ with the mixed complex 
attached to the coalgebra $A^{\textup{!`}}$
$$
(C_\bullet(A),b,B)\simeq
(C_\bullet(A^{\textup{!`}}),
b^*,B^*),
$$
where the right hand is the 
Hochschild complex of the coalgebra 
$A^{\textup{!`}}$
together with the coalgebraic
Hochschild boundary $b^*$ and
Connes cocyclic operator $B^*$.
Since $A^!$ is locally finite, $(C_\bullet(A^{\textup{!`}}), b^*,B^*)$, via linear duality, is exactly the cyclic cochain complex computing the cyclic cohomology of $A^!$.
Passing to homology gives the desired isomorphism.
\end{proof}

\section{Loday's proposal and a Koszul-dual 
model for contravariant 
$K$-theory}\label{sec:loday}

As we recalled in \S\ref{sect:intro},
Loday proposed
the following:

\begin{conjecture}
[{\cite[\S11.4.4]{Loday}}]\label{conj:Loday}
For a dg algebra $A$, there exist $K$-contravariant groups $K^n(A)$ and a Chern character $\mathrm{ch}^*:K^n(A)\to HC^n(A)$ such that the following formula
\begin{equation}\label{conjpairing}
\langle\pi,\xi\rangle_K
=\langle
\mathrm{ch}^*(\pi),\mathrm{ch}_*(\xi)
\rangle_{\mathrm{cyc}}
\end{equation}
holds, where $\langle-,-\rangle_K$ would be a pairing in $K$-theory.
\end{conjecture}

\noindent Note that in this conjecture,
$\langle-,-\rangle_K$
is completely determined by 
the right hand
side of \eqref{conjpairing}, 
namely by  the
Chern character maps
and the pairing on the cyclic 
cohomology
and homology.

We next spell out the construction suggested by the previous sections. In the rest of this section, we assume that $A$ is weakly Adams connected.

\subsection{The candidate groups}

\begin{definition}\label{def:Kdual}
For the above dg algebra $A$, let
\[
K^n_{\mathrm{Kos}}
(A):=K_n (\thick_A(\kfield) )
\text{ for }n\ge 0.
\]
\end{definition}

By Corollary~\ref{cor:$K$-equivalence}, this 
group is canonically isomorphic to 
$K_n(\Adual)$.
The point of this definition is that 
$K^n_{\mathrm{Kos}}(A)$ is constructed from a 
subcategory of the category of $A$-modules, but via Koszul 
duality it is computed by the ordinary 
$K$-theory of the dual algebra $\Adual$.

\begin{definition}\label{def:chdual}
Define the \emph{contravariant Chern character}
\[
\ch_n^{\vee}:K^n_{\mathrm{Kos}}
(A)\longrightarrow HC^n(A)
\]
to be the composition
\[
K_n (\thick_A(\kfield) )
\xrightarrow{\ \sim\ } K_n(\Adual)
\xrightarrow{\ch_n} HC^{-}_n(\Adual)
\xrightarrow{\ \sim\ } HC^n(A).
\]
The first arrow is induced by derived Koszul duality, the middle arrow is the Jones–Goodwillie Chern character, and the last arrow is the isomorphism of Jones--McCleary.
\end{definition}

\begin{proposition}\label{prop:naturality}
The assignment $A\mapsto K^n_{\mathrm{Kos}}
(A)$ is functorial with respect to morphisms 
of augmented dg algebras. The maps $\ch_n^{\vee}$ 
are natural with respect to these morphisms.
\end{proposition}

\begin{proof}
The three maps in Definition~\ref{def:chdual} are natural. 
Indeed, a morphism of augmented dg algebras induces 
a triangulated functor on derived categories, 
hence on 
the thick subcategories generated by the 
augmentation modules. The construction of 
$\Adual$ is functorial up to 
quasi-isomorphism, 
and both the Jones--Goodwillie Chern 
character and the Jones--McCleary comparison 
are natural. 
\end{proof}

\subsection{The pairing}

Let $\eta\in K_n(A)$ and 
$\xi\in K^n_{\mathrm{Kos}}(A)$. Their images 
under the ordinary and contravariant Chern 
characters are classes
\[
\ch_n(\eta)\in HC^-_n(A),
\qquad
\ch_n^{\vee}(\xi)\in HC^n(A).
\]
Using the pairing \eqref{eq:cyclic-pairing}, 
we obtain a bilinear form
\begin{equation}\label{eq:Kpairing}
\langle \xi,\eta\rangle_K
:= \langle\ch_n^{\vee}
(\xi), \ch_n(\eta) \rangle_{\mathrm{cyc}}
\in \kfield.
\end{equation}

\begin{theorem}\label{thm:loday}
The pairing \eqref{eq:Kpairing} is 
characterized by the identity
\[
\langle \xi,\eta\rangle_K
=
 \langle 
\ch_n^{\vee}
(\xi),
\ch_n(\eta) \rangle_{\mathrm{cyc}}\quad\text{for any}\quad\xi\in K^n_{\mathrm{Kos}}(A),\ \eta\in 
K_n(A).
\]
Equivalently, the triangle
\[
\xymatrixcolsep{4pc}
\xymatrix{
K^n_{\mathrm{Kos}}(A)\times K_n(A) 
\ar[r]^-{\ch_n^{\vee}\times \ch_n} 
\ar[dr]_{\langle -, -\rangle_K} & 
HC^n(A)\times HC^-_n(A) 
\ar[d]^{\langle -, -\rangle_{\mathrm{cyc}}} \\
& \kfield
}
\]
commutes. In this sense, the groups 
$K^n_{\mathrm{Kos}}(A)$ realize a concrete 
version of Loday's conjectural contravariant 
$K$-groups.
\end{theorem}

\begin{proof}
This is immediate from 
Definition~\ref{def:chdual} and 
\eqref{eq:Kpairing}.
\end{proof}

We have thus proved Theorem
\ref{thm:intro-main}.

\begin{remark}
The theorem does \emph{not} claim that 
$K^n_{\mathrm{Kos}}(A)$ is the final or 
universal solution to Loday's problem. 
Rather, it shows that with the thick subcategory 
generated by the augmentation module, derived 
Koszul duality supplies a natural and 
computable candidate.
\end{remark}

\subsection{How far is $\thick_A(\kfield)$ from $\Rep(A)$?}
Loday proposed to define $K^n(A)$
as $K_n$ of the category of finite
dimensional representations of $A$.
Since we now assume that $A$ is a dg algebra, it is better to consider the corresponding derived
categories. That is, instead of $\mathrm{Rep}(A)$,
we consider $\mathrm{D}_{\mathrm{fd}}(A)$,
the thick subcategory of $\mathrm{D}(A)$ 
generated by all dg $A$-modules
whose homology is finite-dimensional.

\begin{proposition}[{\cite[Lemma 7.1(a)]{LPWZ}}]
Let $A$ be a strongly locally finite dg algebra.
Then $\mathrm{thick}_A(k)=\mathrm{D}_{\mathrm{fd}}(A)$.
\end{proposition}

Thus, if we combine this proposition
with Theorem
\ref{thm0:koszul-equivalence}, then we get the
following.

\begin{corollary}\label{cor:thickeqfd}
For a strongly locally finite
dg algebra $A$ (and in particular,
for an Adams connected dg algebra $A$),  the groups 
$K^n_{\mathrm{Kos}}(A)$,
built from (the derived category of)
finite-dimensional 
representations of $A$, satisfy the desired property of Conjecture~\text{\rm \ref{conj:Loday}}.
\end{corollary}

\end{document}